\documentclass[11pt,twoside]{article}
\usepackage{mathrsfs}
\usepackage{amsmath}
\usepackage{amsthm}
\usepackage{amsfonts}
\usepackage{amssymb}
\usepackage{color}
\usepackage{latexsym}
\usepackage[all]{xy}

\date{\empty}
\allowdisplaybreaks[4] 

\numberwithin{equation}{section} \theoremstyle{plain}
\newtheorem*{thm*}{Main Theorem}
\newtheorem{theorem}{Theorem}[section]
\newtheorem{corollary}[theorem]{Corollary}
\newtheorem*{corollary*}{Corollary}

\newtheorem*{claim*}{Claim}
\newtheorem{lemma}[theorem]{Lemma}
\newtheorem*{lemma*}{Lemma}
\newtheorem{proposition}[theorem]{Proposition}
\newtheorem*{proposition*}{Proposition}
\newtheorem{remark}[theorem]{Remark}
\newtheorem*{remark*}{Remark}
\newtheorem{example}[theorem]{Example}
\newtheorem*{example*}{Example}

\newtheorem*{question*}{Question}
\newtheorem{definition}[theorem]{Definition}
\newtheorem*{definition*}{Definition}

\newtheorem*{acknowledgements*}{ACKNOWLEDGEMENTS}

\newcommand{\on}{\mathbf{0}}
\newcommand{\uno}{\mathbf{1}}

\DeclareMathOperator{\rk}{rk}
\usepackage{eucal}
\newcommand{\rr}{\mathcal{R}}
\newcommand{\kk}{\mathcal{N}}
\renewcommand{\ss}{\mathcal{S}}
\newcommand{\core}[1]{#1^{\tiny{\textcircled{\tiny\#}}}}
\newcommand{\coep}[1]{#1^{\tiny{\textcircled{\tiny\dag}}}}

\newcommand{\mat}[4]{\left[ \begin{array}{cc} #1 & #2 \\ #3 & #4 \end{array} \right]}
\newcommand{\lie}[2]{\left[ \begin{array}{c}#1 \\ #2 \end{array} \right]}

\newcommand{\vv}[2]{\left[ \begin{array}{c} #1 \\ #2 \end{array} \right]}
\newcommand{\bra}[2]{\langle #1, #2 \rangle}
\newcommand{\pare}[2]{( #1, #2 )}

\begin{document}
\begin{center}
{\large \bf Generalized core inverses of matrices}

\vspace{0.4cm} {\small \bf Sanzhang Xu},
\footnote{ Sanzhang Xu (E-mail: xusanzhang5222@126.com): School of Mathematics, Southeast University, Nanjing 210096, China.}
\vspace{0.4cm} {\small \bf Jianlong Chen},
\footnote{ Jianlong Chen (Corresponding author: jlchen@seu.edu.cn): School of Mathematics, Southeast University, Nanjing 210096, China.}
\vspace{0.4cm} {\small \bf Julio Ben\'{\i}tez}
\footnote{Julio Ben\'{\i}tez (E-mail: jbenitez@mat.upv.es): Universitat Polit\`{e}cnica de Val\`{e}ncia, Instituto de Matem\'{a}tica Multidisciplinar, Valencia, 46022, Spain.}
\vspace{0.4cm} {\small \bf and Dingguo Wang}
\footnote{ Dingguo Wang (E-mail: dingguo95@126.com): School of Mathematical Sciences, Qufu Normal University, Qufu, 273165, China.}

\end{center}
\bigskip

{ \bf  Abstract:}  \leftskip0truemm\rightskip0truemm
In this paper, we introduce two new generalized inverses of matrices, namely,
the $\bra{i}{m}$-core inverse and the $\pare{j}{m}$-core inverse.
The $\bra{i}{m}$-core inverse of a complex matrix extends
the notions of the core inverse defined by Baksalary and Trenkler \cite{BT}
and the core-EP inverse defined by Manjunatha Prasad and Mohana \cite{MM}.
The $\pare{j}{m}$-core inverse of a complex matrix extends
the notions of the core inverse and the ${\rm DMP}$-inverse defined by Malik and Thome \cite{MT}.
Moreover, the formulae and properties of these two new concepts
are investigated by using matrix decompositions and matrix powers.


 \bigskip

\section { \bf Introduction}
Let $\mathbb{C}^{m\times n}$ denote the set of all $m\times n$ complex matrices.
Let $A^{\ast}$, $\rr{(A)}$ and $\rk(A)$ denote the conjugate
transpose, column space, and rank of $A\in \mathbb{C}^{m\times n}$, respectively.
For $A\in \mathbb{C}^{m\times n}$, if $X\in \mathbb{C}^{n\times m}$ satisfies $AXA=A$, $XAX=X$, $(AX)^*=AX$ and $(XA)^*=XA$,
then $X$ is called a {\em Moore-Penrose inverse} of $A$.
This matrix $X$ is unique and denoted by $A^{\dagger}$.
A matrix $X\in \mathbb{C}^{n\times m}$ is called an {\em outer inverse} of $A$ if it satisfies $XAX=X$;
is called a {\em $\{2,3\}$-inverse} of $A$ if it satisfies $XAX=X$ and $(AX)^{\ast}=AX$;
is called a {\em $\{1,3\}$-inverse} of $A$ if it satisfies $AXA=A$ and $(AX)^{\ast}=AX$;
is called a {\em $\{1,2,3\}$-inverse} of $A$ if it satisfies $AXA=A$, $XAX=X$ and $(AX)^{\ast}=AX$.

The core inverse of a complex matrix was introduced by Baksalary and Trenkler \cite{BT}.
Let $A\in \mathbb{C}^{n\times n}$.
A matrix $X\in\mathbb{C}^{n\times n}$ is called a {\em core inverse} of $A$, if it satisfies
$AX=P_{A}$ and $\mathcal{R}(X)\subseteq \mathcal{R}(A)$, here $P_{A}$ denotes the orthogonal projector onto $\mathcal{R}(A)$.
If such a matrix exists, then it is unique and denoted by $A^{\tiny\textcircled{\tiny\#}}$.
For a square complex matrix $A$, one has that $A$ is core invertible, $A$ is group invertible, and
${\rm rk}(A)={\rm rk}(A^2)$ are three equivalent conditions (see [1]). We denote $\mathbb{C}^{CM}_{n}=\{A\in \mathbb{C}^{n\times n}\mid \rk(A)=\rk(A^{2})\}$.

Let $A\in \mathbb{C}^{n\times n}$. A matrix $X\in \mathbb{C}^{n\times n}$ such that
$XA^{k+1}=A^{k}$, $XAX=X$ and $AX=XA$ is called the
{\em Drazin inverse} of $A$ and denoted by $A^{D}$. The smallest integer $k$ is
called the Drazin index of $A$, denoted by ${\rm ind}(A)$. Any square matrix with finite index is Drazin invertible.
If ${\rm ind}(A)\leq 1$, then the Drazin inverse of $A$ is called the {\em group inverse} and denoted by $A^\#.$

The DMP-inverse for a complex matrix was introduced by Malik and Thome \cite{MT}.
Let $A\in \mathbb{C}^{n\times n}$ with ${\rm ind}(A)=k$. A matrix $X\in \mathbb{C}^{n\times n}$ is called a
{\em DMP-inverse} of $A$, if it satisfies
$XAX=X$, $XA=A^{D}A$ and $A^{k}X=A^{k}A^{\dagger}.$
If such a matrix $X$ exists, then it is unique and denoted by $A^{D,\dagger}$.
Malik and Thome gave several characterizations of the DMP-inverse by using the decomposition of Hartwig and Spindelb\"{o}ck \cite{HS}.

The notion of the core-EP inverse for a complex matrix was introduced by Manjunatha Prasad and Mohana \cite{MM}.
A matrix $X\in \mathbb{C}^{n\times n}$ is a {\em  core-EP inverse} of $A\in \mathbb{C}^{n\times n}$ if $X$ is an outer
inverse of $A$ satisfying $\rr{(X)} = \rr{(X^{\ast})} = \rr{(A^{k})}$,
where $k$ is the index of $A$. If such a matrix $X$ exists, then it is unique and denoted by $\coep{A}$.

In addition, $\uno_n$ and $\on_n$ will denote the $n\times 1$ column vectors all of whose components
are $1$ and $0$, respectively. $0_{m\times n}$ (abbr. $0$) denotes the zero matrix of size $m\times n$.
If $\ss$ is a subspace of $\mathbb{C}^{n}$, then $P_{\ss}$ stands for the {\em orthogonal projector} onto the
subspace $\ss$.   A matrix $A\in \mathbb{C}^{n\times n}$ is called an {\em EP} matrix if $\rr{(A)}=\rr{(A^{\ast})}$,
$A$ is called {\em Hermitian} if $A^{\ast}=A$ and $A$ is {\em unitary} if $AA^{\ast}=I_{n},$
where $I_{n}$ denote the {\em identity matrix} of size $n.$
Let $\mathbb{N}$ denotes the set of positive integers.

\section { \bf Preliminaries}

A related decomposition of the matrix decomposition of Hartwig and Spindelb\"{o}ck \cite{HS} was given in \cite[Theorem 2.1]{B} by Ben\'{\i}tez, in \cite{BL} it can be found a simpler proof of this decomposition.
Let us start this section with the concept of principal angles.

\begin{definition}\emph{\cite{W}} \label{pa}
Let $\ss_1$ and $\ss_2$ be two nontrivial subspaces of
$\mathbb{C}^{n}$.
We define the {\em principal angles} $\theta_1, \ldots, \theta_r \in [0,\pi/2]$ between
$\ss_1$ and $\ss_2$ by
$$
\cos\theta_i = \sigma_i(P_{\ss_1}P_{\ss_2}),
$$
for $i=1, \ldots, r$, where $r = \min\{ \dim \ss_1, \dim \ss_2\}$. The real numbers $\sigma_i (P_{\ss_1}P_{\ss_2}) \geq 0$
are the singular values of $P_{\ss_1}P_{\ss_2}$.
\end{definition}

The following theorem can be found in \cite[Theorem 2.1]{B}.
\begin{theorem}  \label{CS}
Let $A\in \mathbb{C}^{n\times n}$, $r = \rk(A)$, and let $\theta_1, \dots, \theta_p$ be the principal
angles between $\rr(A)$ and $\rr(A^*)$ belonging to $]0,\pi/2[$. Denote by $x$ and $y$
the multiplicities of the angles $0$ and $\pi/2$ as a canonical angle between $\rr(A)$
and $\rr(A^*)$, respectively. There exists a unitary matrix $U \in \mathbb{C}^{n\times n}$ such that
\begin{equation}\label{cs}
A = U \left[ \begin{array}{cc} M C & M S \\ 0 & 0 \end{array} \right]U^*,
\end{equation}
where $M \in \mathbb{C}^{r\times r}$ is nonsingular,
$$C = {\rm diag}(\on_y, \cos \theta_1, \ldots, \cos \theta_p, \uno_x),$$
\begin{equation*}
S = \left[ \begin{array}{cc}
{\rm diag} (\uno_y, \sin \theta_1, \ldots, \sin \theta_p) & 0_{p+y,n-(r+p+y)}
 \\ 0_{x,p+y} & 0_{x,n-(r+p+y)}
\end{array} \right],
\end{equation*}
and $r=y+p+x$. Furthermore, $x$ and $y+n-r$ are the multiplicities of the singular
values 1 and 0 in $P_{\rr(A)}P_{\rr(A^*)}$, respectively.
\end{theorem}

In this decomposition, one has $C^2 + SS^*=I_r$.
Recall that $A^\dagger$ always exists. We have that $A^\#$ exists if and only if $C$ is nonsingular in view of \cite[Theorem~3.7]{B}.
The following equalities hold
$$
A^\dag = U \mat{CM^{-1}}{0}{S^* M^{-1}}{0}U^*, \quad
A^\# = U \mat{C^{-1}M^{-1}}{C^{-1}M^{-1}C^{-1}S}{0}{0}U^*.
$$
By \cite[Theorem 2]{BL}, we have that
\begin{equation} \label{csdrazinforl}
A^{D}= U \mat{(MC)^{D}}{[(MC)^{D}]^{2}MS}{0}{0}U^*.
\end{equation}
We also have
\begin{equation}    \label{daacore}
AA^{\dagger} = U \mat{I_{r}}{0}{0}{0}U^*,
\end{equation}
\begin{equation} \label{corecsforma}
\core{A} = A^\#AA^\dag = U \mat{C^{-1}M^{-1}}{0}{0}{0}U^*.
\end{equation}

\begin{lemma} \emph{\cite[Theorem 3.1]{XCZ}} \label{threecore}
Let $A\in \mathbb{C}^{n\times n}$. Then $A$ is core invertible if and only if there exists $X\in \mathbb{C}^{n\times n}$ such that $(AX)^{\ast}=AX$, $XA^{2}=A$ and $AX^{2}=X$.
In this situation, we have $A^{\tiny{\textcircled{\tiny\#}}}=X$.
\end{lemma}

\begin{lemma}\label{twedlema}
Let $A\in \mathbb{C}^{n\times n}$. If there exists $X\in \mathbb{C}^{n\times n}$ such that $AX^{k+1}=X^{k}$ and $XA^{k+1}=A^{k}$ for some $k\in \mathbb{N}$,
then  for $m\in\mathbb{N}$ we have
\begin{itemize}
\item[{\rm (1)}] $A^{k}=X^{m}A^{k+m}$;
\item[{\rm (2)}] $X^{k}=A^{m}X^{k+m}$;
\item[{\rm (3)}] $A^{k}X^{k}=A^{k+m}X^{k+m}$;
\item[{\rm (4)}] $X^{k}A^{k}=X^{k+m}A^{k+m}$;
\item[{\rm (5)}] $A^{k}=A^{m}X^{m}A^{k}$;
\item[{\rm (6)}] $X^{k}=X^{m}A^{m}X^{k}$.
\end{itemize}
\end{lemma}
\begin{proof}
$(1)$. For $m = 1$, it is clear by the hypotheses. If the formula is true for $m\in \mathbb{N}$, then
$X^{m+1}A^{k+m+1} =XX^{m}A^{k+m}A = XA^{k}A = XA^{k+1} = A^{k}$.

$(3)$. It is easy to check that $A^{k}X^{k}=A^{k+1}X^{k+1}$ by $AX^{k+1}=X^{k}$. It is not difficult to check the equality $A^{k}X^{k}=A^{k+m}X^{k+m}$ by induction.

$(5)$. From $(1)$ we have $A^k=X^kA^{2k}$. Thus by $AX^{k+1}=X^{k}$, we have
$A^k=X^kA^{2k}=AX^{k+1}A^{2k}=AX^kXA^{2k}=A(AX^{k+1})XA^{2k}=A^2X^{k+2}A^{2k}=A^2X^2X^kA^{2k}
=\cdots=A^mX^mX^kA^{2k}=A^mX^mA^k$.

The proofs of $(2)$, $(4)$ and $(6)$ are similar to the proofs of $(1)$, $(3)$ and $(5)$, respectively.
\end{proof}

\begin{lemma}\label{drazinlema}
Let $A\in \mathbb{C}^{n\times n}$. If there exists $X\in \mathbb{C}^{n\times n}$ such that $AX^{k+1}=X^{k}$ and $XA^{k+1}=A^{k}$ for some $k\in \mathbb{N}$,
then $A^{D}=X^{k+1}A^{k}$.
\end{lemma}
\begin{proof}
Since $A$ is Drazin invertible by $XA^{k+1}=A^{k}$, we will check that $A^{D}=X^{k+1}A^{k}$.
Have in mind, $AX^{k+1}=X^{k}$ and $XA^{k+1}=A^{k}$, thus
\begin{equation} \label{drazinllla}
A(X^{k+1}A^{k})=X^{k}A^{k}=X^{k}(XA^{k+1})=X^{k+1}A^{k}A.
\end{equation}
That is, $X^{k+1}A^{k}$ and $A$ are commute. Then by $(1)$ and $(4)$ in Lemma \ref{twedlema}, we have that
\begin{equation} \label{drazinlllb}
\begin{split}
(X^{k+1}A^{k})A(X^{k+1}A^{k})=&X^{k+1}A^{k+1}X^{k+1}A^{k}=X^{k}A^{k}(X^{k+1}A^{k})\\
=&X^{k}X^{k+1}A^{k}A^{k}=X^{k+1}X^{k}A^{2k}=X^{k+1}A^{k}.
\end{split}
\end{equation}
From $(1)$ in Lemma \ref{twedlema}, we have that
\begin{equation} \label{drazinlllc}
(X^{k+1}A^{k})A^{k+1}=X(X^{k}A^{2k})A=XA^{k+1}=A^{k}.
\end{equation}
Thus we have $A^{D}=X^{k+1}A^{k}$ by the definition of the Drazin inverse and in view of $(\ref{drazinllla})$, $(\ref{drazinlllb})$ and $(\ref{drazinlllc})$.
\end{proof}

\begin{remark}
From the proofs of Lemma~$\ref{twedlema}$ and Lemma~$\ref{drazinlema}$, it is obvious that
Lemma~$\ref{twedlema}$ and Lemma~$\ref{drazinlema}$ are valid for rings.
Moreover, we can get that for an element $a\in R$, $a$ is Drazin invertible if and only if
there exist $x\in R$ and $k\in \mathbb{N}$ such that $ax^{k+1}=x^{k}$ and $xa^{k+1}=a^{k}$, where $R$ is a ring.
\end{remark}

The following lemma is similar to \cite[Theorem 2.5]{MT}. 
\begin{lemma} \label{coepcsfoemula}
Let $A\in \mathbb{C}^{n\times n}$ be the form \emph{(\ref{cs})}. Then
\begin{equation} \label{usecsformula}
A^{D,\dagger}=U\mat{(MC)^{D}}{0}{0}{0}U^{\ast}.
\end{equation}
\end{lemma}
\begin{proof}
By (\ref{csdrazinforl}) and (\ref{daacore}), we can get
$A^{D}= U \mat{(MC)^{D}}{[(MC)^{D}]^{2}MS}{0}{0}U^*$ and $AA^{\dagger} = U \mat{I_{r}}{0}{0}{0}U^*$, respectively.
Thus by the definition of DMP-inverse we have
\begin{equation*}
A^{D,\dagger}=A^{D}AA^{\dagger}=U \mat{(MC)^{D}}{[(MC)^{D}]^{2}MS}{0}{0}\mat{I_{r}}{0}{0}{0}U^*
=U\mat{(MC)^{D}}{0}{0}{0}U^*.
\end{equation*}
\end{proof}

\begin{lemma} \emph{\cite[Corollary 3.3]{WA}}\label{procorp}
Let $A\in \mathbb{C}^{n\times n}$ be a matrix of index $k$. Then $A\coep{A}=A^{k}(A^{k})^{\dagger}.$
\end{lemma}

\section { \bf $\bra{i}{m}$-core inverse}
Let us start this section by introducing the definition of the $\bra{i}{m}$-core inverse.

\begin{definition} \label{kcoredefj}
Let $A\in \mathbb{C}^{n\times n}$ and $m,i\in\mathbb{N}$. A matrix $X\in \mathbb{C}^{n\times n}$ is called an
$\bra{i}{m}$-core inverse of $A$, if it satisfies
\begin{equation} \label{kcoredefij}
X=A^{D}AX~~\text{and}~~A^{m}X=A^{i}(A^{i})^{\dagger}.
\end{equation}
\end{definition}

It will be proved that if $X$ exists, then it is unique and denoted by $A_{i,m}^{\oplus}$.

\begin{theorem} \label{formulaja}
Let $A\in \mathbb{C}^{n\times n}$. If exists $X\in \mathbb{C}^{n\times n}$ such that \emph{(\ref{kcoredefij})} holds, then $X$ is unique.
\end{theorem}
\begin{proof}
Assume that $X$ satisfies the system in (\ref{kcoredefij}), that is $X=A^{D}AX$ and $A^{m}X=A^{i}(A^{i})^{\dagger}$. Thus
$X=A^{D}AX=(A^{D})^{m}A^{m}X=(A^{D})^{m}A^{i}(A^{i})^{\dagger}.$ Therefore, $X$ is unique by the uniqueness of $A^{D}$ and $A^{i}(A^{i})^{\dagger}$.
\end{proof}

\begin{theorem} \label{formulaj}
The system in \emph{(\ref{kcoredefij})} is consistent if and only if $i\geq {\rm ind}(A)$.
In this case, the solution of \emph{(\ref{kcoredefij})} is $X=(A^{D})^{m}A^{i}(A^{i})^{\dagger}$.
\end{theorem}
\begin{proof} Assume that  $i\geq {\rm ind}(A)$.
Let $X=(A^{D})^{m}A^{i}(A^{i})^{\dagger}$. We have
\begin{eqnarray*}
&&A^DAX=A^DA(A^{D})^{m}A^{i}(A^{i})^{\dagger}=(A^{D})^{m}A^DAA^{i}(A^{i})^{\dagger}=(A^{D})^{m}A^{i}(A^{i})^{\dagger}=X;\\
&&A^mX=A^m(A^{D})^{m}A^{i}(A^{i})^{\dagger}=A^DAA^{i}(A^{i})^{\dagger}=A^{i}(A^{i})^{\dagger}.
\end{eqnarray*}
Thus, the system in (\ref{kcoredefij}) is consistent and the solution of (\ref{kcoredefij}) is $X=(A^{D})^{m}A^{i}(A^{i})^{\dagger}$.

If the system in (\ref{kcoredefij}) is consistent, then exists $X_{0}$ such that
$X_{0}=A^DAX_{0}$ and $A^m X_{0}=A^i (A^{i})^{\dagger}$. Then $X_{0}=A^DAX_{0}=(A^D)^{m}A^mX_{0}=(A^{D})^{m}A^{i}(A^{i})^{\dagger}$
and $A^i (A^{i})^{\dagger}=A^m X_{0}=A^m (A^{D})^{m}A^{i}(A^{i})^{\dagger}=AA^DA^i (A^{i})^{\dagger}$.
Hence $A^i=A^i (A^{i})^{\dagger} A^i=AA^DA^i (A^{i})^{\dagger} A^{i}=AA^DA^i$, that is $i\geq {\rm ind}(A)$.
\end{proof}

\begin{example} \label{exmpmja}
We will give an example that shows if $i< {\rm ind}(A)$, then the system in \emph{(\ref{kcoredefij})} is not consistent.
Let $A=\mat{0}{1}{0}{0}$. It is easy to get ${\rm ind}(A)=2$ and $A^D=0$. Let $i=1$ and suppose that $X$ is the solution of system in $(\ref{kcoredefij})$,
then $X=A^DAX=0$, which gives $AA^\dag=A^mX=0$, thus $A=AA^\dag A=0$, this is a contradiction.
\end{example}

\begin{remark} \label{mjjiangci}
If $i\geq {\rm ind}(A)$, then $A_{i,m+1}^{\oplus}=A^DA_{i,m}^{\oplus}$.
\end{remark}

\begin{remark} \label{mjgenerthree}
The $\bra{i}{m}$-core inverse is a generalization of the core inverse and the core-EP inverse.
More precisely, we have the following statements:
\begin{itemize}
\item[{\rm (1)}] If $m=i={\rm ind}(A)=1$, then the ${\rm \langle 1,1\rangle}$-core inverse coincides with the core inverse;
\item[{\rm (2)}] If $m=1$ and $i={\rm ind}(A)$, then the ${\rm \langle i,1\rangle}$-core inverse coincides with the core-EP inverse.
\end{itemize}
\end{remark}
For the convenience of the readers, in the following, we give some notes of $(1)$ and $(2)$ in Remark~\ref{mjgenerthree}.

$(1)$. If $m=i={\rm ind}(A)=1$, then $A$ is group invertible and $A^{D}=A^\#$ and $(\ref{kcoredefij})$ is equivalent to
$X=A^\#AX$ and $AX=AA^{\dagger}$. Thus $X=A^\#AX=A^\#AA^{\dagger},$ $(AX)^{\ast}=(AA^{\dagger})^{\ast}=AA^{\dagger}=AX$,
$AX^{2}=AA^\#AA^{\dagger}A^\#AA^{\dagger}=AA^\#AA^{\dagger}AA^\#A^{\dagger}=AA^\#A^{\dagger}=X$ and
$XA^{2}=A^\#AA^{\dagger}A^{2}=A^\#A^{2}=A$. Hence, ${\rm \langle 1,1\rangle}$-core inverse coincides with the core inverse by Lemma~\ref{threecore}.
Note that if $A$ is group invertible, then we have that $X$ is the core inverse of $A$ if and only if $X=A^\# AX$ and $AX=AA^\dag$.

$(2)$. If $m=1$ and $i={\rm ind}(A)$, then by Theorem 3.3, $A_{i,1}^\oplus$ exists and
$A_{i,1}^\oplus = A^D A^i (A^i)^\dagger$. Let us denote
$X= A_{i,1}^\oplus = A^D A^i (A^i)^\dagger$. Observe that
$AX = A^i (A^i)^\dag$ is Hermitian. Now,
$$
X A X
=
A^D A^i (A^i)^\dagger A^i (A^i)^\dagger
=
A^D A^i (A^i)^\dagger
=
X,
$$
that is $X$ is an outer inverse of $A$.
From $A^i = A^D A^{i+1} = A^D A^i (A^i)^\dag A^i A = X A^{i+1}$
we get $\mathcal{R}(A^i) \subseteq \mathcal{R}(X)$. Also,
$A X^2 = (AX)X=A^i (A^i)^\dag A^D A^i (A^i)^\dag =
A^i (A^i)^\dag A^i A^D (A^i)^\dag = A^D A^i (A^i)^\dag = X$,
which implies $X = (A X)^*X \in \mathcal{R}(X^*)$, therefore,
$\mathcal{R}(X) \subseteq \mathcal{R}(X^*)$. Finally,
$X^* = [A^D A^i (A^i)^\dag]^* = A^i (A^i)^\dag (A^D)^*$ implies
$\mathcal{R}(X^*) \subseteq \mathcal{R}(A^i)$. Hence
$\mathcal{R}(X) = \mathcal{R}(X^*) = \mathcal{R}(A^i)$.
 Therefore, the ${\rm \langle i,1\rangle}$-core inverse coincides with the core-EP inverse
by the definition of the core-EP inverse.

From the above statement, we have the following theorem.

\begin{theorem}
Let $A\in \mathbb{C}^{n\times n}$ with $i={\rm ind}(A)$. Then $X$ is the core-EP inverse of $A$ if and only if
$X=A^D AX$ and $AX=A^i (A^i)^{\dagger}$.
\end{theorem}

\begin{corollary}
Let $A\in \mathbb{C}^{n\times n}$ with $1={\rm ind}(A)$. Then $X$ is the core inverse of $A$ if and only if
$X=A^\# AX$ and $AX=AA^{\dagger}$.
\end{corollary}

For any $A\in \mathbb{C}^{n\times n}$, either $A^{l}=0$ for some $l\in \mathbb{N}$, or $A^{l}\neq 0$ for all positive integers.
Moreover, if ${\rm ind}(A)=k$, then $G_{k}B_{k}$ is nonsingular (see \cite{Ch,Cj,Cg}), where
$A=B_{1}G_{1}$ is a full rank factorization of $A$ and $G_{l}B_{l}=B_{l+1}G_{l+1}$ is a full rank factorization of $G_{l}B_{l}$, $l=1,\ldots,k-1$.
When $A^{k}\neq 0$, then it can be written as
\begin{equation} \label{repreak}
A^{k}=\prod^{k}_{l=1}B_{l}\prod^{k}_{l=1}G_{k+1-l}.
\end{equation}
We have the following results, (see \cite[Theorem 4]{Ch} or
\cite[Theorem 7.8.2]{CM}):
\begin{equation*}
{\rm ind}(A)=\left\{\begin{aligned}
&k,  &&{}\text{when}~G_{k}B_{k}~\text{is~nonsingualr};\\
&k+1,               &&{}\text{when}~G_{k}B_{k}=0.
\end{aligned}
\right.
\end{equation*}
and
\begin{equation} \label{drazinrepreak}
A^{D}=\left\{\begin{aligned}
&\prod^{k}_{l=1}B_{l}(G_{k}B_{k})^{-k-1}\prod^{k}_{l=1}G_{k+1-l},  &&{}\text{when}~G_{k}B_{k}~\text{is~nonsingualr};\\
&0,               &&{}\text{when}~G_{k}B_{k}=0.
\end{aligned}
\right.
\end{equation}
In the sequel, we always assume that $A^{k}\neq 0$.

It is well-known that if $A=EF$ is a full rank factorization of $A$, where $r=\rk(A)$, $E\in \mathbb{C}^{n\times r}$ and $F\in \mathbb{C}^{r\times n}$, then
(see \cite[Theorem 1.3.2]{CM})
\begin{equation} \label{mpfullrank}
A^{\dagger}=F^\ast(FF^\ast)^{-1}(E^\ast E)^{-1}E^\ast.
\end{equation}

\begin{remark} \label{remakpower}
The notations and  results in above paragraph will be used many times in the sequel.
\end{remark}

We will investigate the $\bra{i}{m}$-core inverse of a matrix $A\in \mathbb{C}^{n\times n}$ by using Remark~\ref{remakpower}.

\begin{theorem} \label{indexsame}
Let $A\in \mathbb{C}^{n\times n}$ with ${\rm ind}(A)=k$. If $i\geq k$, then $A_{i,m}^{\oplus}=A_{k,m}^{\oplus}$.
\end{theorem}
\begin{proof}
Since ${\rm ind}(A)=k$, we have $\mathcal{R}(A^k)=\mathcal{R}(A^i)$ for any $i \geq k$, and
therefore, $A^k (A^k)^\dag = A^i (A^i)^\dag$. Now, the conclusion follows from Theorem~\ref{formulaj}.
\end{proof}

\begin{remark} \label{indexsameaa}
The proof of Theorem~$\ref{indexsame}$ also can be proved as follows. Since the proof in this remark will be used several times in the sequel,
we write this proof here.
\end{remark}
\begin{proof}
If $A$ is nilpotent, then $A^D=0$, hence
by Theorem~\ref{formulaj}, one has $A^\oplus_{i,m} = A^\oplus_{k,m}=0$. Therefore, we can assume that $A^k \neq 0$.
By equality (\ref{repreak}), we have
\begin{equation} \label{repreakaa}
A^{k}=\prod^{k}_{l=1}B_{l}\prod^{k}_{l=1}G_{k+1-l}.
\end{equation}
where
$A=B_{1}G_{1}$ is a full rank factorization of $A$ and $G_{l}B_{l}=B_{l+1}G_{l+1}$ is a full rank factorization of $G_{l}B_{l}$, $l=1,\ldots,k-1$.
Let $M=\prod\limits_{l=1}^{k}B_{l}$, $N=\prod\limits_{l=1}^{k}G_{k+1-l}$ and $L=G_{k}B_{k}$. Now, we will show that
$$A^{i}=\prod\limits^{k}_{l=1}B_{l}(G_{k}B_{k})^{i-k}\prod\limits^{k}_{l=1}G_{k+1-l}=ML^{i-k}N.$$ In fact,
\begin{equation} \label{ajrepresata}
\begin{split}
A^{i}=&\prod^{i}_{l=1}B_{l}\prod^{i}_{l=1}G_{k+1-l}=B_{1}\cdots B_{i}G_{i}\cdots G_{1}
= B_{1}\cdots B_{i-1}(B_{i}G_{i})G_{i-1}\cdots G_{1}\\
=& B_{1}\cdots B_{i-1}(G_{i-1}B_{i-1})G_{i-1}\cdots G_{1}
= B_{1}\cdots B_{i-2}(G_{i-2}B_{i-2})^{2}G_{i-2}\cdots G_{1}\\
=& \cdots\cdots\\
=& B_{1}\cdots B_{k}(G_{k}B_{k})^{i-k}G_{k}\cdots G_{1}
= ML^{i-k}N.
\end{split}
\end{equation}
If we let $M_{1}= ML^{i-k}$, then $A^{i}=ML^{i-k}N=M_{1}N$ is a full rank factorization of $A^i$ (see \cite[p.183]{Cg}). Thus
\begin{equation} \label{mpfullrankaa}
(A^{i})^{\dagger}=N^\ast(NN^\ast)^{-1}(M_{1}^\ast M_{1})^{-1}M_{1}^\ast.
\end{equation}
Note that $NM=\prod\limits_{l=1}^{k} G_{k+1-l}\prod\limits_{l=1}^{k} B_{l}=L^{k}$. By Theorem~\ref{formulaj}, (\ref{drazinrepreak}) and $(\ref{mpfullrankaa})$ we have
\begin{equation} \label{ajrepresatadd}
\begin{split}
A^{\oplus}_{i,1}=&A^{D}A^{i}(A^{i})^{\dagger}=ML^{-k-1}NML^{i-k}N(A^{i})^{\dagger}\\
=& ML^{-k-1}NML^{i-k}NN^\ast(NN^\ast)^{-1}(M_{1}^\ast M_{1})^{-1}M_{1}^\ast\\
=& ML^{i-k-1}NN^\ast(NN^\ast)^{-1}(M_{1}^\ast M_{1})^{-1}M_{1}^\ast\\
=& ML^{i-k-1}(M_{1}^\ast M_{1})^{-1}M_{1}^\ast\\
=& ML^{i-k-1}[(L^{i-k})^\ast M^{\ast}ML^{i-k}]^{-1}(L^{i-k})^\ast M^{\ast}\\
=& ML^{i-k-1}L^{k-i}(M^{\ast}M)^{-1}[(L^{i-k})^\ast]^{-1}(L^{i-k})^\ast M^{\ast}\\
=& ML^{-1}(M^{\ast}M)^{-1}M^{\ast}.
\end{split}
\end{equation}
The last expression does not depend on $i$, then
$A^\oplus_{i,1} = A^\oplus_{k,1}$.
Thus, by Remark~$\ref{mjjiangci}$, we have
$
A^{\oplus}_{i,m}= A^DA^{\oplus}_{i,m-1}=A^D(A^DA^{\oplus}_{i,m-2})=(A^D)^{2}A^{\oplus}_{i,m-2}
=\cdots =(A^D)^{m-1}A^{\oplus}_{i,1}=(A^D)^{m-1}A^{\oplus}_{k,1}=A^{\oplus}_{k,m}.
$
\end{proof}

\begin{remark}
By Theorem~$\ref{indexsame}$, it is enough to investigate the $i={\rm ind}(A)=k$ case, when we discuss
the $\bra{i}{m}$-core inverse of a matrix $A\in \mathbb{C}^{n\times n}$. That is, the Theorem~$\ref{indexsame}$ is a key theorem.
\end{remark}

\begin{theorem} \label{indexsameformu}
Let $A\in \mathbb{C}^{n\times n}$ with ${\rm ind}(A)=k$ and $k,m\in\mathbb{N}$. If $A=B_{1}G_{1}$ is a full rank factorization of $A$
and $G_{l}B_{l}=B_{l+1}G_{l+1}$ is a full rank factorization of $G_{l}B_{l}$, $l=1,\ldots,k-1$, then $A_{k,m}^{\oplus}=ML^{-m}M^{\dagger}$,
where $M=\prod\limits_{l=1}^{k}B_{l}$, $N=\prod\limits_{l=1}^{k}G_{k+1-l}$ and $L=G_{k}B_{k}$.
\end{theorem}
\begin{proof}
By the proof of Remark~\ref{indexsameaa}, we have $A^{\oplus}_{k,1}=ML^{-1}(M^\ast M)^{-1}M^{\ast}$ and $NM=L^{k}$.
Now, we will prove $(A^D)^{s}A^{\oplus}_{k,1}=ML^{-s-1}(M^\ast M)^{-1}M^{\ast}$ for any $s\in \mathbb{N}$.
By (\ref{drazinrepreak}) we have $A^{D}=\prod\limits^{k}_{l=1}B_{l}(G_{k}B_{k})^{-k-1}\prod\limits^{k}_{l=1}G_{k+1-l}=ML^{-k-1}N$.
When $s=1$, we have
\begin{equation*}
\begin{split}
A^DA^{\oplus}_{k,1}=& ML^{-k-1}NML^{-1}(M^\ast M)^{-1}M^{\ast}
= ML^{-k-1}(NM)L^{-1}(M^\ast M)^{-1}M^{\ast}\\
=& ML^{-k-1}L^{k}L^{-1}(M^\ast M)^{-1}M^{\ast}
= ML^{-2}(M^\ast M)^{-1}M^{\ast}.
\end{split}
\end{equation*}
Assume that $(A^D)^{s-1}A^{\oplus}_{k,1}=ML^{-s}(M^\ast M)^{-1}M^{\ast}$. Then
\begin{equation*}
\begin{split}
(A^D)^{s}A^{\oplus}_{k,1}=& A^D(A^D)^{s-1}A^{\oplus}_{k,1}
= A^DML^{-s}(M^\ast M)^{-1}M^{\ast}\\
=& ML^{-k-1}NML^{-s}(M^\ast M)^{-1}M^{\ast}
= ML^{-k-1}L^kL^{-s}(M^\ast M)^{-1}M^{\ast}\\
=& ML^{-s-1}(M^\ast M)^{-1}M^{\ast}.
\end{split}
\end{equation*}
Thus by Remark~\ref{mjjiangci}, we have
$$A^{\oplus}_{k,m}=(A^D)^{m-1}A^{\oplus}_{k,1}=ML^{-m}(M^{\ast}M)^{-1}M^{\ast}=ML^{-m}M^{\dagger}.$$
\end{proof}

In the following theorem, we will give a canonical form for the $\bra{k}{m}$-core inverse of a matrix $A\in \mathbb{C}^{n\times n}$
by using the matrix decomposition in Theorem~\ref{CS}.
We will also use the following simple fact:
Let $X \in \mathbb{C}^{n \times m}$ and ${\bf b} \in \mathbb{C}^n$.
If ${\bf y} \in \mathbb{C}^m$ satisfies $X^*X {\bf y}= X^* {\bf b}$, then
$XX^\dag {\bf b} = X{\bf y}$.

\begin{theorem} \label{csfoemulajj}
Let $A\in \mathbb{C}^{n\times n}$ have the form (\emph{\ref{cs}}) with ${\rm ind}(A)=k$ and $m\in \mathbb{N}$. Then
\begin{equation} \label{usecsformulaj}
A^{\oplus}_{k,m}=U\mat{(MC)^{\oplus}_{k-1,m}}{0}{0}{0}U^{\ast}.
\end{equation}
\end{theorem}
\begin{proof}
Let $r$ be the rank of $A$. By Theorem~\ref{formulaj} we have
$$ 
A^\oplus_{k,m} = (A^D)^m A^k (A^k)^\dag.
$$ 
Since $A$ has the form given in Theorem~\ref{CS} we have
\begin{equation}\label{tres}
A^k = U \mat{(MC)^k}{(MC)^{k-1}MS}{0}{0} U^*.
\end{equation}
Let ${\bf b} \in \mathbb{C}^n$ be arbitrary and let us decompose ${\bf b} = U\vv{{\bf b}_1}{{\bf b}_2}$,
where ${\bf b}_1 \in \mathbb{C}^r$. Let ${\bf x}_0 \in \mathbb{C}^n$ satisfy
$(A^k)^* A^k {\bf x}_0 = (A^k)^* {\bf b}$ [this ${\bf x}_0$ always exists because
the normal equations always have a solution]. We can decompose
${\bf x}_0$ by writing ${\bf x}_0 = U \vv{{\bf x}_1}{{\bf x}_2}$, where ${\bf x}_1 \in \mathbb{C}^r$.
Let us denote $N=(MC)^{k-1}M$. Using (\ref{tres}),
$$
U \mat{CN^*}{0}{S^*N^*}{0} \mat{NC}{NS}{0}{0} \vv{{\bf x}_1}{{\bf x}_2} =
U \mat{CN^*}{0}{S^*N^*}{0} \vv{{\bf b}_1}{{\bf b}_2}.
$$
Therefore,
$$
CN^*N(C {\bf x}_1+S{\bf x}_2) = CN^* {\bf b}_1 \qquad \text{and} \qquad
S^*N^*N(C {\bf x}_1+S{\bf x}_2) = S^*N^* {\bf b}_1.
$$
Premultiplying the first equality by $C$ and the second equality by $S$ and after, adding them,
we get $N^*N(C {\bf x}_1+S{\bf x}_2) = N^* {\bf b}_1$, and hence,
$N(C {\bf x}_1+S{\bf x}_2) = NN^\dag {\bf b}_1$. Now,
\begin{equation*}
\begin{split}
A^k(A^k)^\dag {\bf b} & = A^k {\bf x}_0 =
U \mat{NC}{NS}{0}{0} \vv{{\bf x}_1}{{\bf x}_2} \\
& = U \vv{NC {\bf x}_1 + NS{\bf x}_2}{{\bf 0}}
= U \vv{NN^\dag {\bf b}_1}{{\bf 0}} =
U \mat{NN^\dag}{0}{0}{0} U^* {\bf b}.
\end{split}
\end{equation*}
Since ${\bf b}$ is arbitrary,
$$
A^k(A^k)^\dag  = U \mat{NN^\dag}{0}{0}{0} U^*.
$$
Now we will prove $N N^\dag = (MC)^{k-1} [(MC)^{k-1}]^\dag$. Recall that we have
$N = (MC)^{k-1}M$, and so, $\mathcal{R}(N) \subseteq \mathcal{R}((MC)^{k-1})$. Since
$M$ is nonsingular, ${\rm rk}(N)={\rm rk}((MC)^{k-1})$, and
therefore, $\mathcal{R}(N) = \mathcal{R}((MC)^{k-1})$. Since
$NN^\dagger$ and $(MC)^{k-1} [(MC)^{k-1}]^\dag$ are the orthogonal projectors onto
$\mathcal{R}(N)$ and $\mathcal{R}((MC)^{k-1})$, respectively, we get
$NN^\dagger = (MC)^{k-1} [(MC)^{k-1}]^\dag$.

By (2.2) we have
\begin{equation}\label{cuatro}
A^D = U \mat{(MC)^D}{[(MC)^D]^2 MS}{0}{0} U^*.
\end{equation}

Thus, we have
$$
(A^D)^m = U \mat{[(MC)^D]^m}{[(MC)^D]^{m+1} MS}{0}{0} U^*.
$$
Since ${\rm ind}(A)=k$, we have $A^D A^{k+1}=A^k$. By using the above representations of
$A^D$ and $A^k$ given in (\ref{tres}) and (\ref{cuatro}), respectively,
$$
\mat{(MC)^D}{[(MC)^D]^2 MS}{0}{0} \mat{(MC)^{k+1}}{(MC)^kMS}{0}{0}
= \mat{(MC)^k}{(MC)^{k-1}MS}{0}{0}.
$$
Therefore,
\begin{equation}\label{cinco}
(MC)^D (MC)^k M [C \ | \ S] = (MC)^{k-1} M[ C \ | \ S].
\end{equation}
Have in mind that we have $C^2+SS^*=I_r$. Thus, postmultiplying (\ref{cinco}) by $\vv{C}{S^*}$
gives us $(MC)^D (MC)^k M = (MC)^{k-1} M$ and from the nonsningularity of $M$ we obtain
$(MC)^D (MC)^k = (MC)^{k-1}$, and so, ${\rm ind}(MC) \leq k-1$.
Therefore we have
\begin{equation*}
\begin{split}
A^{\oplus}_{k,m}=&(A^{D})^{m}A^{k}(A^{k})^{\dagger}\\
=&U \mat{[(MC)^{D}]^{m}}{[(MC)^{D}]^{m+1}MS}{0}{0}\mat{(MC)^{k-1}((MC)^{k-1})^{\dagger}}{0}{0}{0}U^*\\
=&U \mat{[(MC)^{D}]^{m}(MC)^{k-1}((MC)^{k-1})^{\dagger}}{0}{0}{0}U^*\\
=&U\mat{(MC)^{\oplus}_{k-1,m}}{0}{0}{0}U^{\ast}.
\end{split}
\end{equation*}
\end{proof}

\begin{remark}
If we use the decomposition of Hartwig and Spindelb\"{o}ck in \cite[Corollary 6]{HS}, then an expression of the $\bra{k}{m}$-core inverse of $A$
is $A^{\oplus}_{k,m}=U\mat{(\Sigma K)^{\oplus}_{k-1,m}}{0}{0}{0}U^{\ast}$, which is similar to the expression of $A^{\oplus}_{k,m}$ in Theorem~$\ref{csfoemulajj}$.
Since the proof of this result can be proved as the proof of
Theorem~$\ref{csfoemulajj}$, we omit this proof.
\end{remark}

Let $A\in \mathbb{C}^{n\times n}$ with ${\rm ind}(A)=k$. The Jordan Canonical form of $A$ is $P^{-1}AP=J$,
where $P\in \mathbb{C}^{n\times n}$ is nonsingular and $J\in \mathbb{C}^{n\times n}$ is a block diagonal matrix composed of Jordan blocks.
In the following theorem, we will compute the $\bra{k}{m}$-core inverse by using the Jordan Canonical form of $A$.

\begin{theorem} \label{jodanbiao}
Let $A\in \mathbb{C}^{n\times n}$ with ${\rm ind}(A)=k$, then $A^{\oplus}_{k,m}=P_{1}D^{-m}P_{1}^{\dagger}$, where
$A=P\mat{D}{0}{0}{N}P^{-1}$ with $D\in \mathbb{C}^{r\times r}$ is nonsingular, $N$ is nilpotent and $P=[P_{1} \ | \ P_{2}]$ with $P_{1}\in \mathbb{C}^{n\times r}$.
\end{theorem}
\begin{proof}
The Jordan Canonical form of $A$ is $P^{-1}AP=J$,
where $P\in \mathbb{C}^{n\times n}$ is nonsingular and $J\in \mathbb{C}^{n\times n}$ is a block diagonal matrix.
Rearrange the elements of $J$ such that $A=P\mat{D}{0}{0}{N}P^{-1}$, where $D$ is nonsingular and $N$ is nilpotent.
It is well-known that $A^D=P\mat{D^{-1}}{0}{0}{0}P^{-1}$ and $A^{k}=P\mat{D^k}{0}{0}{0}P^{-1}$. If we let $P=[P_{1} \ | \ P_{2}]$
and $P^{-1}=\lie{Q_{1}}{Q_{2}}$, then
$(A^D)^{m}A^{k}=[P_{1} \ | \ P_{2}]\mat{(D^{-1})^{m}}{0}{0}{0}\mat{D^{k}}{0}{0}{0}\lie{Q_{1}}{Q_{2}}=P_{1}D^{k-m}Q_{1}$.
Observe that $A^k = (P_1D^k)Q_1$ is a full rank factorization
of $A^k$. Hence by (\ref{mpfullrank}) we have
\begin{equation*}
\begin{split}
(A^{k})^{\dagger}=&(P_{1}D^kQ_{1})^{\dagger}=Q_{1}^{\ast}(Q_{1}Q_{1}^{\ast})^{-1}[(P_{1}D^k)^{\ast}P_{1}D^k]^{-1}(P_{1}D^k)^{\ast}\\
=&Q_{1}^{\ast}(Q_{1}Q_{1}^{\ast})^{-1}D^{-k}(P_{1}^{\ast}P_{1})^{-1}[(D^k)^{\ast}]^{-1}(D^k)^{\ast}P_{1}^{\ast}\\
=&Q_{1}^{\ast}(Q_{1}Q_{1}^{\ast})^{-1}D^{-k}(P_{1}^{\ast}P_{1})^{-1}P_{1}^{\ast}\\
=&Q_{1}^{\dagger}D^{-k}P_{1}^{\dagger}.
\end{split}
\end{equation*}
By Theorem \ref{formulaj}, we have $A^{\oplus}_{k,m}=(A^{D})^{m}A^{k}(A^{k})^{\dagger}$. Thus we have
\begin{equation*}
\begin{split}
A^{\oplus}_{k,m}=&(A^{D})^{m}A^{k}(A^{k})^{\dagger}
=P_{1}D^{k-m}Q_{1}Q_{1}^{\dagger}D^{-k}P_{1}^{\dagger}
=P_{1}D^{k-m}Q_{1}Q_{1}^{\ast}(Q_{1}Q_{1}^{\ast})^{-1}D^{-k}P_{1}^{\dagger}\\
=&P_{1}D^{-m}D^{k}D^{-k}P_{1}^{\dagger}
=P_{1}D^{-m}P_{1}^{\dagger}.
\end{split}
\end{equation*}
\end{proof}

\begin{proposition}\label{propjectormkcorej}
Let $A\in \mathbb{C}^{n\times n}$. If $i\geq {\rm ind}(A)$, then $A^mA^{\oplus}_{i,m}$ is the projector onto $\rr{(A^{i})}$ along $\rr{(A^{i})}^{\bot}$.
\end{proposition}
\begin{proof}
It is trivial.
\end{proof}

In the following proposition, we will investigate some properties of the $\bra{i}{m}$-core inverse.
\begin{proposition}\label{proposiofmkcorej}
Let $A\in \mathbb{C}^{n\times n}$, $m,i\in\mathbb{N}$. If $i\geq {\rm ind}(A)$, then
\begin{itemize}
\item[{\rm (1)}] $A^{\oplus}_{i,m}$ is a $\{2,3\}$-inverse of $A^{m}$;
\item[{\rm (2)}] $A^{\oplus}_{i,m}=(A^{D})^{m}P_{A^{i}}$;
\item[{\rm (3)}] $(A^{\oplus}_{i,m})^{n}=(A^{D})^{m(n-1)}P_{A^{i}}$;
\item[{\rm (4)}] $A^{i}A^{\oplus}_{i,m}=A^{\oplus}_{i,m}A^{i}$ if and only if $\rr{(A^{i})}^{\bot}\subseteq \kk{((A^{D})^{m})}$;
\item[{\rm (5)}] $A^{\oplus}_{i,m}=A$ implies that $A$ is EP.
\end{itemize}
\end{proposition}
\begin{proof}
$(1)$. By Theorem~\ref{formulaj} we have $A^{\oplus}_{i,m}=(A^{D})^{m}A^{i}(A^{i})^{\dagger}$, thus
\begin{equation*}
\begin{split}
A^{\oplus}_{i,m}A^{m}A^{\oplus}_{i,m}=&(A^{D})^{m}A^{i}(A^{i})^{\dagger}A^{m}(A^{D})^{m}A^{i}(A^{i})^{\dagger}
=(A^{D})^{m}A^{i}(A^{i})^{\dagger}A^{i}A^{m}(A^{D})^{m}(A^{i})^{\dagger}\\
=&(A^{D})^{m}A^{i}A^{m}(A^{D})^{m}(A^{i})^{\dagger}
=(A^{D})^{m}A^{m}(A^{D})^{m}A^{i}(A^{i})^{\dagger}\\
=&A^{D}A(A^{D})^{m}A^{i}(A^{i})^{\dagger}
=(A^{D})^{m}A^{i}(A^{i})^{\dagger}
=A^{\oplus}_{i,m}.
\end{split}
\end{equation*}
Thus $A^{\oplus}_{i,m}$ is a $\{2,3\}$-inverse of $A^{m}$ in view of $A^{m}A^{\oplus}_{i,m}=A^{i}(A^{i})^{\dagger}$.

$(2)$ is trivial.

$(3)$. By $(A^{\oplus}_{i,m})^{2}=(A^{D})^{m}A^{i}(A^{i})^{\dagger}(A^{D})^{m}A^{i}(A^{i})^{\dagger}
=(A^{D})^{m}(A^{D})^{m}A^{i}(A^{i})^{\dagger}=(A^{D})^{m}A^{\oplus}_{i,m}$ and induction it is easy to check $(3)$.

$(4)$.  By $\rr{(I_{n}-A^{i}(A^{i})^{\dagger})}=\kk{((A^{i})^{\dagger})}$, we have
\begin{equation*}
\begin{split}
A^{i}A^{\oplus}_{i,m}=A^{\oplus}_{i,m}A^{i}~~
\Leftrightarrow&~~A^{i}(A^{D})^{m}A^{i}(A^{i})^{\dagger}=(A^{D})^{m}A^{i}(A^{i})^{\dagger}A^{i}\\
\Leftrightarrow&~~A^{i}(A^{D})^{m}A^{i}(A^{i})^{\dagger}=(A^{D})^{m}A^{i}\\
\Leftrightarrow&~~A^{i}(A^{D})^{m}(I_{n}-A^{i}(A^{i})^{\dagger})=0\\
\Leftrightarrow&~~\rr{(I_{n}-A^{i}(A^{i})^{\dagger})}\subseteq \kk{(A^{i}(A^{D})^{m})}\\
\Leftrightarrow&~~\kk{((A^{i})^{\dagger})}\subseteq \kk{((A^{D})^{m})}\\
\Leftrightarrow&~~\kk{((A^{i})^{\ast})}\subseteq \kk{((A^{D})^{m})}\\
\Leftrightarrow&~~\rr{(A^{i})}^{\bot}\subseteq \kk{((A^{D})^{m})}.
\end{split}
\end{equation*}

$(5)$. Let $A$ be written in the form $(\ref{cs})$.
We have $A^{\oplus}_{i,m}=U\mat{(MC)^{\oplus}_{i-1,m}}{0}{0}{0}U^{\ast}$ by Theorem~\ref{csfoemulajj}.
Thus, $A^{\oplus}_{i,m}=A$ implies $MS=0.$ From the nonsingularity of $M$, we have
$S=0$, which is equivalent to say that $A$ is EP in view of \cite[Theorem 3.7]{B}.

\end{proof}
\section { \bf $\pare{j}{m}$-core inverse}
Let us start this section by introducing the definition of the $\pare{j}{m}$-core inverse.

\begin{definition} \label{kcoredef}
Let $A\in \mathbb{C}^{n\times n}$ and $m,j\in\mathbb{N}$. A matrix $X\in \mathbb{C}^{n\times n}$ is called a
$\pare{j}{m}$-core inverse of $A$, if it satisfies
\begin{equation} \label{kcoredefi}
X=A^{D}AX~~\text{and}~~A^{m}X=A^{m}(A^{j})^{\dagger}.
\end{equation}
\end{definition}

\begin{theorem} \label{uniquemkcorea}
Let $A\in \mathbb{C}^{n\times n}$. If the system in \emph{(\ref{kcoredefi})} is consistent, then the solution is unique.
\end{theorem}
\begin{proof}
Assume that $X$ satisfies that (\ref{kcoredefi}), that is $X=A^{D}AX$ and $A^{m}X=A^{m}(A^{j})^{\dagger}$. Then
$X=A^{D}AX=(A^{D})^{m}A^{m}X=(A^{D})^{m}A^{m}(A^{j})^{\dagger}=A^{D}A(A^{j})^{\dagger}.$ Thus $X$ is unique.
\end{proof}

By Theorem~\ref{uniquemkcorea} if $X$ exists, then it is unique and denoted by $A^{\ominus}_{j,m}$.

\begin{theorem} \label{uniquemkcore}
Let $A\in \mathbb{C}^{n\times n}$ and $m,j\in\mathbb{N}$. Then
\begin{itemize}
\item[{\rm (1)}] If $m\geq {\rm ind}(A)$, then the system in \emph{(\ref{kcoredefi})} is consistent and the solution is $X=A^{D}A(A^{j})^{\dagger}$;
\item[{\rm (2)}] If the system in \emph{(\ref{kcoredefi})} is consistent, then ${\rm ind}(A)\leq \max\{j,m\}$.
\end{itemize}
\end{theorem}
\begin{proof}
$(1)$. Let $X=A^{D}A(A^{j})^{\dagger}$. We have $A^DAX=A^DAA^DA(A^{j})^{\dagger}=A^DA(A^{j})^{\dagger}=X$ and
$A^mX=A^mA^DA(A^{j})^{\dagger}=A^DAA^m(A^{j})^{\dagger}=A^{m}(A^{j})^{\dagger}$.

$(2)$.  If the system in (\ref{kcoredefi}) is consistent, then exits $X_{0}\in \mathbb{\mathbb{C}}^{n\times n}$
such that $X_{0}=A^{D}AX_{0}=(A^{D})^{m}A^{m}X_{0}=(A^{D})^{m}A^{m}(A^{j})^{\dagger}=A^{D}A(A^{j})^{\dagger}$ and
$A^{m}(A^{j})^{\dagger}=A^{m}X_{0}=A^{m}A^{D}A(A^{j})^{\dagger}=A^{m}(A^{D})^{j}A^{j}(A^{j})^{\dagger}$.
Thus $$A^m(A^{j})^{\dagger}A^{j}=A^{m}(A^{D})^{j}A^{j}(A^{j})^{\dagger}A^{j}=A^{m}(A^{D})^{j}A^{j}=A^{m}A^{D}A.$$
If $m\geq j$, then $A^{m}A^{D}A=A^m(A^{j})^{\dagger}A^{j}=A^{m-j}A^{j}(A^{j})^{\dagger}A^{j}=A^{m-j}A^{j}=A^{m}$. That is ${\rm ind}(A)\leq m$.
If $j> m$, then $A^j=A^{j}(A^{j})^{\dagger}A^{j}=A^{j-m}A^m(A^{j})^{\dagger}A^{j}=A^{j-m}A^{m}A^{D}A=A^j A^D A$. That is ${\rm ind}(A)\leq j$.
Therefore, ${\rm ind}(A)\leq \max\{j,m\}$.
\end{proof}

\begin{example}
We will give an example that shows if $m< {\rm ind}(A)$, then the system in $(\ref{kcoredefi})$ is not consistent.
Let $A$ be the same matrix in Example~$\ref{exmpmja}$. It is easy to get ${\rm ind}(A)=2$ and $A^D=0$. Let $m=j=1$
and suppose that $X$ is the solution of system in $\ref{kcoredefi}$, then $X=A^DAX=0$, which gives $AA^\dag=AX=0$,
thus $A=AA^\dag A=0$, this is a contradiction.
\end{example}

\begin{example}
The converse of Theorem~$\ref{uniquemkcore}$ $(1)$ is not true.
Let $m=1$ and $j=3$. If we let
$A=
\left[ \begin{array}{ccc}
0               & 1             &     0 \\
0               & 0             &     1\\
0               & 0             &     0
\end{array} \right]$,
then ${\rm ind}(A)=3$ and $A^3=0$. Hence $X=0$ is a solution of $(\ref{kcoredefi})$, but $m< {\rm ind}(A)$.
\end{example}

\begin{example}
If ${\rm ind}(A)\leq \max\{j,m\}$, then the system in $(\ref{kcoredefi})$ may be not consistent.
If we let
$A=
\left[ \begin{array}{ccc}
2               & 2             &     1 \\
-1              & -1            &     0\\
0               & 0             &     0
\end{array} \right]$,
then
$A^3=A^2=
\left[ \begin{array}{ccc}
2               & 2             &     2 \\
-1              & -1            &     -1\\
0               & 0             &     0
\end{array} \right]$,
$A^D=A^{2}$ and ${\rm ind}(A)=2$.
Let $m=1$ and $j=2$, then ${\rm ind}(A)\leq \max\{j,m\}$.
It is easy to check that
$(A^2)^{\dagger}=\frac{1}{15}
\left[ \begin{array}{ccc}
2               & -1             &     0\\
2               & -1             &     0\\
2               & -1             &     0
\end{array} \right]$.
If the system in $(\ref{kcoredefi})$ has a solution $X_{0}$, then
$X_{0}=A^DAX_{0}=A^DA(A^2)^{\dagger}$ and $A(A^2)^{\dagger}=AX_{0}=AA^DA(A^2)^{\dagger}=A^{4}(A^2)^{\dagger}=A^{2}(A^2)^{\dagger}$ would hold. But
$A(A^2)^{\dagger}=\frac{1}{15}
\left[ \begin{array}{ccc}
10               & -5             &     0\\
-4               & 2              &     0\\
0                & 0              &     0
\end{array} \right]
\neq
\frac{1}{15}
\left[ \begin{array}{ccc}
12               & -6             &     0\\
-6               & 3              &     0\\
0                & 0              &     0
\end{array} \right]=A^{2}(A^2)^{\dagger}$.
Thus, the system in $(\ref{kcoredefi})$ is not consistent.
\end{example}

\begin{remark}
If $m\geq {\rm ind}(A)=k$, it is not difficult to see that $A^{\ominus}_{j,m}=A^{\ominus}_{j,m+1}$.
That is to say, the $\pare{j}{m}$-core inverse of $A$ coincides with the $\pare{j}{m+1}$-core inverse of $A$.
Thus, in the sequel, we only discuss the $m={\rm ind}(A)$ case.
\end{remark}

\begin{theorem} \label{equavalentcona}
Let $A, X\in \mathbb{C}^{n\times n}$, $k,j\in\mathbb{N}$. If ${\rm ind}(A)=k$
and $X$ is the $\pare{j}{k}$-core inverse of $A$, then we have $X^{j}A^{j}X^{j}=(A^{D})^{j(j-1)}X^{j}$ and $XA^{j}=A^{D}A$.
\end{theorem}
\begin{proof}
By the definition of the $(j,k)$-core inverse,
we have $X=A^{D}AX$ and $A^{k}X=A^{k}(A^{j})^{\dagger}$.
By $X=A^{D}A(A^{j})^{\dagger}$, it is easy to check that $X^{n+1}=(A^{D})^{j}X^{n}$ for arbitrary $n\in \mathbb{N}$, which gives that
$X^{j}=(A^{D})^{j(j-1)}X.$
\begin{equation*}
\begin{split}
XA^{j}=&A^{D}A(A^{j})^{\dagger}A^{j}=(A^{D})^{j}A^{j}(A^{j})^{\dagger}A^{j}=(A^{D})^{j}A^{j}=A^{D}A;\\
X^{j}A^{j}X^{j}=&(A^{D})^{j(j-1)}XA^{j}X^{j}=(A^{D})^{j(j-1)}A^{D}A(A^{j})^{\dagger}A^{j}X^{j}\\
=&(A^{D})^{j(j-1)}(A^{D})^{j}A^{j}(A^{j})^{\dagger}A^{j}X^{j}=(A^{D})^{j(j-1)}(A^{D})^{j}A^{j}X^{j}\\
=&(A^{D})^{j(j-1)}A^{D}AXX^{j-1}=(A^{D})^{j(j-1)}A^{D}AA^{D}A(A^{j})^{\dagger}X^{j-1}\\
=&(A^{D})^{j(j-1)}A^{D}A(A^{j})^{\dagger}X^{j-1}=(A^{D})^{j(j-1)}XX^{j-1}\\
=&(A^{D})^{j(j-1)}X^{j}.
\end{split}
\end{equation*}
\end{proof}

\begin{corollary} \label{equavalentconaa}
Let $A, X\in \mathbb{C}^{n\times n}$ and ${\rm ind}(A)=k$.
If $X$ is the $(1,k)$-core inverse of $A$, then we have $XAX=X$ and $XA=A^{D}A$.
\end{corollary}

The $(j,m)$-core inverse is a generalization of the core inverse and the
DMP-inverse in view of Theorem~\ref{equavalentcona}.

\begin{remark}
When $j=m=1={\rm ind}(A)$, the equations in $(\ref{kcoredefi})$ are equivalent to $XAX=X$, $XA=A^\#A$ and $AX=AA^{\dagger}$.
Thus $AX=AA^{\dagger}$ implies that $(AX)^{\ast}=AX$; $XA=A^\#A$ gives that $XA^{2}=A$ and $AXA=A$; and
$X=XAX=A^\#AX=AA^\#$, which means that $\rr{(X)}\subseteq \rr{(A)}$, then $X=AY$ for some $Y\in \mathbb{C}^{n\times n}$,
thus $X=AY=AXAY=AX^{2}$. Therefore, we have $\core{A}=X$ by Lemma~$\ref{threecore}$. In a word, the $(1,1)$-core inverse
coincides with the usual core inverse.
\end{remark}

\begin{remark} \label{ramarkmk}
If we let $j=1$ and $m={\rm ind}(A)$, then the equations in $(\ref{kcoredefi})$ are equivalent to
$XAX=X$, $XA=A^{D}A$ and $A^{k}X=A^{k}A^{\dagger}$ by Theorem~$\ref{equavalentcona}$.
Thus $(1,k)$-core inverse coincides with the DMP-inverse.
\end{remark}

From Remark~\ref{ramarkmk},  Theorem~\ref{equavalentcona} and the definition of the $(j,k)$-core inverse, we have the following theorem,
which says that the conditions $XAX=X$ and $XA=A^D A$ in the definition of the DMP-inverse can be replaced by $X=A^D AX$.

\begin{theorem}
Let $A\in \mathbb{C}^{n\times n}$ with $k={\rm ind}(A)$. Then $X\in \mathbb{C}^{n\times n}$ is the DMP-inverse of $A$
if and only if $X=A^D AX$ and $A^{k}X=A^{k}A^{\dagger}$.
\end{theorem}

In the following theorem, we will give a canonical form for the $\pare{j}{k}$-core inverse of a matrix $A\in \mathbb{C}^{n\times n}$
by using the matrix decomposition in Theorem~\ref{CS}.

\begin{theorem} \label{csfoemula}
Let $A\in \mathbb{C}^{n\times n}$ have the form \emph{(\ref{cs})} with ${\rm ind}(A)=k$ and $j\in \mathbb{N}$. Then
\begin{equation} \label{usecsformula}
A^{\ominus}_{j,k}=U\mat{(MC)^{D}(MC)^{\ominus}_{j-1,k}}{0}{0}{0}U^{\ast}.
\end{equation}
\end{theorem}
\begin{proof}
By Theorem~\ref{uniquemkcore} and the idempotency of $A^{D}A$ we have
\begin{equation} \label{kcoreqeuaa}
A^{\ominus}_{j,k}=A^{D}A(A^{j})^{\dagger}=(A^{D})^{j}A^{j}(A^{j})^{\dagger}.
\end{equation}
From the proof of Theorem~\ref{csfoemulajj}, we have
\begin{equation} \label{kcoreqeuai}
A^{j}(A^{j})^{\dagger}=U\mat{(MC)^{j-1}((MC)^{j-1})^{\dagger}}{0}{0}{0}U^{\ast}.
\end{equation}
By (\ref{csdrazinforl}) we have
$A^{D}= U \mat{(MC)^{D}}{[(MC)^{D}]^{2}MS}{0}{0}U^*$, thus we have
\begin{equation} \label{kcoreqeuaj}
(A^{D})^{j}=U \mat{[(MC)^{D}]^{j}}{[(MC)^{D}]^{j+1}MS}{0}{0}U^*.
\end{equation}
By the proof of Theorem~\ref{csfoemulajj}, we have ${\rm ind}(MC)\leq k-1< k$.
From (\ref{kcoreqeuaa}), (\ref{kcoreqeuai}) and (\ref{kcoreqeuaj}), we have
\begin{equation*}
\begin{split}
A^{\ominus}_{j,k}=&(A^{D})^{j}A^{j}(A^{j})^{\dagger}\\
=&U \mat{[(MC)^{D}]^{j}}{[(MC)^{D}]^{j+1}MS}{0}{0}\mat{(MC)^{j-1}((MC)^{j-1})^{\dagger}}{0}{0}{0}U^*\\
=&U \mat{[(MC)^{D}]^{j}(MC)^{j-1}((MC)^{j-1})^{\dagger}}{0}{0}{0}U^*\\
=&U \mat{(MC)^{D}[(MC)^{D}]^{j-1}(MC)^{j-1}((MC)^{j-1})^{\dagger}}{0}{0}{0}U^*\\
=&U \mat{(MC)^{D}(MC)^{D}MC((MC)^{j-1})^{\dagger}}{0}{0}{0}U^*\\
=&U\mat{(MC)^{D}(MC)^{\ominus}_{j-1,k}}{0}{0}{0}U^{\ast}.
\end{split}
\end{equation*}
\end{proof}

\begin{remark}
If we use the decomposition of Hartwig and Spindelb\"{o}ck in \cite[Corollary 6]{HS}, then an expression of the $\pare{j}{k}$-core inverse of $A$
is $A^{\ominus}_{j,k}=U\mat{(\Sigma K)^{D}(\Sigma K)^{\ominus}_{j-1,k}}{0}{0}{0}U^{\ast}$, which is similar to the expression of $A^{\ominus}_{j,k}$ in Theorem~$\ref{csfoemula}$.
Since the proof of this result can be proved like the proof of
Theorem~~$\ref{csfoemula}$, we omit this proof.
\end{remark}

\begin{theorem} \label{csfoemulabb}
Let $A\in \mathbb{C}^{n\times n}$ and ${\rm ind}(A)=k$. If $(A^{k}X^{k})^{\ast}=A^{k}X^{k}$, $AX^{k+1}=X^{k}$ and $XA^{k+1}=A^{k}$, then
$A$ is (k,k)-core invertible and $A^{\ominus}_{k,k}=X^{k}.$
\end{theorem}
\begin{proof}
By Lemma~\ref{twedlema} and Lemma~\ref{drazinlema}, we have
$A^{k}X^{k}A^{k}=A^{k}$, $X^{k}A^{k}X^{k}=X^{k}$, $A^{k}=X^{k}A^{2k}$ and $A^{D}=X^{k+1}A^{k}$.
Equalities $(A^{k}X^{k})^{\ast}=A^{k}X^{k}$ and $A^{k}X^{k}A^{k}=A^{k}$ imply that $X^{k}$ is a $\{1,3\}$-inverse of $A^{k}$.
From $A^{D}=X^{k+1}A^{k}$, we can obtain $(A^{D})^{k}=X^{k-1}A^{D}$ by induction.
Thus
\begin{equation*}
\begin{split}
A^{\ominus}_{k,k}=&A^{D}A(A^{k})^{\dagger}=(A^{D})^{k}A^{k}(A^{k})^{\dagger}=(A^{D})^{k}A^{k}(A^{k})^{(1,3)}\\
=&(A^{D})^{k}A^{k}X^{k}=(X^{k+1}A^{k})^{k}A^{k}X^{k}=X^{k-1}X^{k+1}A^{k}A^{k}X^{k}\\
=&X^{2k}A^{2k}X^{k}=X^{k}(X^{k}A^{2k})X^{k}=X^{k}A^{k}X^{k}=X^{k}.
\end{split}
\end{equation*}
\end{proof}

\begin{proposition}\label{propjectormkcore}
Let $A\in \mathbb{C}^{n\times n}$ be a matrix with $j\geq {\rm ind}(A)=k$. If $A$ is $\pare{j}{k}$-core invertible, then $A^jA^{\ominus}_{j,k}$ is the
projector onto $\rr{(A^{j})}$ along $\rr{(A^{j})}^{\bot}$.
\end{proposition}
\begin{proof}
It is trivial.
\end{proof}

In the following proposition, we will investigate some properties of the $\pare{j}{k}$-core inverse.
\begin{proposition}\label{proposiofmkcore}
Let $A\in \mathbb{C}^{n\times n}$ with $j\geq {\rm ind}(A)=k$. If $A$ is $\pare{j}{k}$-core invertible, then
\begin{itemize}
\item[{\rm (1)}] $A^{\ominus}_{j,k}$ is a $\{1,2,3\}$-inverse of $A^{j}$;
\item[{\rm (2)}] $A^{\ominus}_{j,k}=(A^{D})^{j}P_{A^{j}}$;
\item[{\rm (3)}] $(A^{\ominus}_{j,k})^{n}=\left\{\begin{array}{l}
                   ((A^{D})^{j}(A^{j})^{\dag})^{\frac{n}{2}}\ \ \ \ \ \ \qquad n\mbox{~is~even}\\
                    A^{j}((A^{D})^{j}(A^{j})^{\dag})^{\frac{n+1}{2}}  \qquad n\mbox{~is~odd}
                 \end{array}
               \right.$;
\item[{\rm (4)}] $A^{\ominus}_{j,k}A^{D}=(A^{D})^{j+1}$;
\item[{\rm (5)}] $A^{j}A^{\ominus}_{j,k}=A^{\ominus}_{j,k}A^{j}$ if and only if $\rr{(A^{j})}^{\bot}\subseteq \kk{(A^{D})}$;
\item[{\rm (6)}] $A^{\ominus}_{j,k}=A$ implies that $A$ is EP.
\end{itemize}
\end{proposition}
\begin{proof}
$(1)$. By Theorem~\ref{uniquemkcore} we have $A^{\ominus}_{j,k}=A^{D}A(A^{j})^{\dagger}=(A^{D})^{j}A^{j}(A^{j})^{\dagger}$, thus
\begin{equation*}
\begin{split}
A^{j}A^{\ominus}_{j,k}A^{j}=&A^{j}(A^{D})^{j}A^{j}(A^{j})^{\dagger}A^{j}=A^{j}(A^{D})^{j}A^{j}=A^{j}A^{D}A=A^{j};\\
A^{\ominus}_{j,k}A^{j}A^{\ominus}_{j,k}=&(A^{D})^{j}A^{j}(A^{j})^{\dagger}A^{j}A^{\ominus}_{j,k}=A^{D}AA^{\ominus}_{j,k}\\
=&A^{D}AA^{D}A(A^{j})^{\dagger}=A^{D}A(A^{j})^{\dagger}=A^{\ominus}_{j,k};\\
A^{j}A^{\ominus}_{j,k}=&A^{j}(A^{D})^{j}A^{j}(A^{j})^{\dagger}=A^{j}(A^{j})^{\dagger}.
\end{split}
\end{equation*}

$(2)$ is trivial.

$(3)$. By $(A^{\ominus}_{j,k})^{2}=(A^{D})^{j}A^{j}(A^{j})^{\dagger}(A^{D})^{j}A^{j}(A^{j})^{\dagger}=(A^{D})^{j}(A^{j})^{\dagger}$ and induction
it is easy to check $(3)$.

$(4)$. $A^{\ominus}_{j,k}A^{D}=(A^{D})^{j}A^{j}(A^{j})^{\dagger}A^{D}=(A^{D})^{j}A^{j}(A^{j})^{\dagger}(A^{D})^{j}A^{j}A^{D}=(A^{D})^{j+1}$.

$(5)$.  By $\rr{(I_{n}-A^{j}(A^{j})^{\dagger})}=\kk{((A^{j})^{\dagger})}$ and $\kk{(A^{D}A)}=\kk{(A^{D})}$, we have
\begin{equation*}
\begin{split}
A^{j}A^{\ominus}_{j,k}=A^{\ominus}_{j,k}A^{j}~~
\Leftrightarrow&~~A^{j}(A^{D})^{j}A^{j}(A^{j})^{\dagger}=(A^{D})^{j}A^{j}(A^{j})^{\dagger}A^{j}\\
\Leftrightarrow&~~A^{j}(A^{D})^{j}A^{j}(A^{j})^{\dagger}=(A^{D})^{j}A^{j}\\
\Leftrightarrow&~~A^{j}(A^{D})^{j}(I_{n}-A^{j}(A^{j})^{\dagger})=0\\
\Leftrightarrow&~~\rr{(I_{n}-A^{j}(A^{j})^{\dagger})}\subseteq \kk{(A^{D}A)}\\
\Leftrightarrow&~~\kk{((A^{j})^{\dagger})}\subseteq \kk{(A^{D}A)}\\
\Leftrightarrow&~~\kk{((A^{j})^{\ast})}\subseteq \kk{(A^{D})}\\
\Leftrightarrow&~~\rr{(A^{j})}^{\bot}\subseteq \kk{(A^{D})}.
\end{split}
\end{equation*}

$(6)$. Let $A$ be written in the form $(\ref{cs})$.
We have $A^{\ominus}_{j,k}=U\mat{(MC)^{D}(MC)^{\ominus}_{j-1,k}}{0}{0}{0}U^{\ast}$ by Theorem~\ref{csfoemula}.
Thus, $A^{\ominus}_{j,k}=A$ implies $MS=0.$ From the nonsingularity of $M$, we have
$S=0$, which is equivalent to say that $A$ is EP in view of \cite[Theorem 3.7]{B}.
\end{proof}

In the following proposition, we shall give the the relationship between the $\pare{j}{k}$-core inverse and DMP-inverse and core-EP inverse.
\begin{proposition} \label{csfoemulabb}
Let $A\in \mathbb{C}^{n\times n}$ with ${\rm ind}(A)=k$. Then $A^{\ominus}_{k,k}=A^{D,\dagger}(A^{D})^{k-1}A\coep{A}.$
\end{proposition}
\begin{proof}
We have that $A^{k}(A^{k})^{\dagger}=A\coep{A}$ by Lemma~\ref{procorp} and $A^{D,\dagger}=A^DAA^\dag$.
Thus \begin{equation*}
\begin{split}
A^{\ominus}_{k,k}=&A^DA(A^{k})^{\dagger}=(A^D)^{k}A^k(A^{k})^{\dagger}=A^DA^k(A^D)^{k-1}(A^{k})^{\dagger}\\
=&A^DAA^\dag A^k(A^D)^{k-1}(A^{k})^{\dagger}=A^{D,\dagger}(A^{D})^{k-1}A^{k}(A^{k})^{\dagger}\\
=&A^{D,\dagger}(A^{D})^{k-1}A\coep{A}.
\end{split}
\end{equation*}
\end{proof}

In the following theorem, we will give a relationship between the $\bra{i}{m}$-core inverse and $\pare{j}{m}$-core inverse.
\begin{theorem} \label{refship}
Let $A\in \mathbb{C}^{n\times n}$ with ${\rm ind}(A)=k$. Then $A^{\oplus}_{k,m}=A^{\ominus}_{m,k}$ for any $m\geq k$.
\end{theorem}
\begin{proof}
By Theorem~\ref{uniquemkcore}, we have
$A^{\ominus}_{m,k}=A^DA(A^{m})^{\dagger}=(A^D)^{k}A^{k}(A^{m})^{\dagger}$.
By the proof of Remark~\ref{indexsameaa}, we have $A^{k}=MN$ and $NM=L^{k}$, where $M=\prod\limits_{l=1}^{k}B_{l}$, $N=\prod\limits_{l=1}^{k}G_{k+1-l}$ and $L=G_{k}B_{k}$.
It is easy to see that $(A^{D})^{s}=ML^{-k-s}N$ for any $s\in \mathbb{N}$ by $NM=L^{k}$. Thus $(A^{D})^{k}=ML^{-2k}N$ and
$$(A^{D})^{k}A^{k}=ML^{-2k}NMN=ML^{-2k}L^{k}N=ML^{-k}N.$$
By the proof of Remark~\ref{indexsameaa}, we have $A^{m}=ML^{m-k}N=M_{1}N$ is a full rank factorization of $A^m$, where $M_{1}= ML^{m-k}$ and
$(A^{m})^{\dagger}=N^\ast(NN^\ast)^{-1}(M_{1}^\ast M_{1})^{-1}(M_{1})^\ast$.
By Theorem~\ref{indexsameformu}, we have  $A_{k,m}^{\oplus}=ML^{-m}M^{\dagger}$. In the following steps, we will show that $A^{\ominus}_{m,k}=ML^{-m}M^{\dagger}$.
From $A^{\ominus}_{m,k}=(A^D)^{k}A^{k}(A^{m})^{\dagger}$, we have
\begin{equation*}
\begin{split}
A^{\ominus}_{k,m}=&(A^D)^{k}A^{k}(A^{m})^{\dagger}=ML^{-k}NN^\ast(NN^\ast)^{-1}(M_{1}^\ast M_{1})^{-1}(M_{1})^\ast\\
=&ML^{-k}(M_{1}^\ast M_{1})^{-1}(M_{1})^\ast
=ML^{-k}[(L^{m-k})^\ast M^{\ast}ML^{m-k}]^{-1}(L^{m-k})^\ast M^{\ast}\\
=&ML^{-k}L^{k-m}(M^{\ast}M)^{-1}[(L^{m-k})^\ast]^{-1}(L^{m-k})^\ast M^{\ast}\\
=& ML^{-m}(M^{\ast}M)^{-1}M^{\ast}=ML^{-m}M^{\dagger}.
\end{split}
\end{equation*}
\end{proof}

\begin{theorem}
Let $A\in \mathbb{C}^{n\times n}$ with $i\geq {\rm ind}(A)=k$, then $A^{\ominus}_{i,k}=P_{1}D^{-i}P_{1}^{\dagger}$, where
$A=P\mat{D}{0}{0}{N}P^{-1}$ with $D\in \mathbb{C}^{r\times r}$ is nonsingular, $N$ is nilpotent and $P=[P_{1} \ | \ P_{2}]$ with $P_{1}\in \mathbb{C}^{n\times r}$.
\end{theorem}
\begin{proof}
It is easy to see that by Theorem~\ref{jodanbiao} and Theorem~\ref{refship}.
\end{proof}

\centerline {\bf ACKNOWLEDGMENTS} This research is supported by the National Natural Science Foundation of China (No. 11371089 and No. 11471186), the Natural Science Foundation of Jiangsu Province (No. BK20141327). The first author is grateful to China Scholarship Council for giving him a purse for his further study in Universitat Polit\`{e}cnica de Val\`{e}ncia, Spain.

\end{document}